\documentclass[11pt]{article}
\usepackage{amsmath, amssymb, eucal, latexsym, multicol, amsthm, bm, placeins}
\usepackage{ifpdf}
\usepackage{xcolor}
\usepackage{epsfig}
\usepackage{bm}
\usepackage{stmaryrd}
\usepackage{fullpage}
\usepackage{enumerate}

\newtheorem{theorem}{Theorem}[section]

\newtheorem{proposition}[theorem]{Proposition}

\theoremstyle{definition}
\newtheorem{definition}[theorem]{Definition}

\theoremstyle{remark}
\newtheorem{remark}[theorem]{Remark}

\numberwithin{equation}{section}

\newcommand{\Z}{\mathbb Z}

\newcommand{\cA}{\mathcal{A}}
\newcommand{\cD}{\mathcal{D}}

\newcommand{\bA}{\mathbf{A}}

\newcommand{\tvert}{\vert\!\vert\!\vert}

\newcommand{\bv}{\mathbf{v}}
\newcommand{\bF}{\mathbf{f}}

\newcommand{\bu}{\mathbf{u}}

\newcommand{\tol}{{\varepsilon}}

\newenvironment{algotab}%
{\par\begin{samepage}%
\begin{tabbing}\ttfamily%
 \hspace*{5mm}\=\hspace{3ex}\=\hspace{3ex}\=\hspace{3ex}\=\hspace{3ex}%
\=\hspace{3ex}\=\hspace{3ex}\=\hspace{3ex}\=\hspace{3ex}\kill}%
{\end{tabbing}\end{samepage}}

\newcommand{\wt}{\widetilde}

\newcommand{\ps}{{\phi^*}}

\date{\today}

\title{{\bf Adaptive Spectral Galerkin Methods \\ with Dynamic Marking}}
\author{
Claudio~Canuto\thanks{Dipartimento di Scienze Matematiche,
Politecnico di Torino,
Corso Duca degli Abruzzi 24,
I-10129 Torino, Italy (claudio.canuto@polito.it )} 
\and
Ricardo~H.~Nochetto%
\thanks{Department of Mathematics and Institute for Physical Science
and Technology, University of Maryland, College Park, Maryland 20742, USA
(rhn@math.umd.edu)}
\and Rob~Stevenson%
\thanks{Korteweg-de Vries Institute for Mathematics,
University of Amsterdam,
P.O. Box 94248,
1090 GE Amsterdam, The Netherlands
(r.p.stevenson@uva.nl)}
\and Marco~Verani\thanks{MOX-Dipartimento di Matematica, Politecnico di Milano, P.zza Leonardo Da Vinci 32, I-20133 Milano, Italy (marco.verani@polimi.it).}}

\begin{document}
\maketitle

\abstract{
The convergence and optimality theory of adaptive Galerkin methods is
almost exclusively based on the D\"orfler marking. This entails a fixed
parameter and leads to a contraction constant bounded below away from
zero. For spectral
Galerkin methods this is a severe limitation which affects performance. We
present a dynamic marking strategy that allows for a super-linear
relation between consecutive discretization errors, and show 
exponential convergence with linear computational
  complexity whenever the solution belongs to a Gevrey
    approximation class.
}

\section{Introduction} 

The modern analysis of adaptive discretizations of partial
differential equations aims at establishing rigorous results of
\emph{convergence} and \emph{optimality}. The former results
concern the convergence of the approximate solutions produced by the
successive iterations of the adaptive algorithm towards the exact
solution $u$, with an estimate of the error decay rate measured
in an appropriate norm. On the other hand, optimality results
compare the cardinality of the active set of basis functions used to expand
the discrete solution to the minimal cardinality
needed to approximate the exact solution with similar accuracy;
this endeavor borrows ideas from Nonlinear Approximation
theory. Confining ourselves in the sequel to second-order elliptic
boundary value problems, such kind of analysis has been carried out
first for wavelet discretizations \cite{CDDV:98,
  Gantumur-Stevenson:2007}, then for $h$-type finite elements
\cite{MNS:00, Stevenson:07,CaKrNoSi:08,CFPP:14,DKS:15}, \cite{MNS:00} dealing just
with convergence, and more recently for spectral-type methods
\cite{CaNoVe:14a, CaNoVe:14b, CaSiVe:14-2}; we refer to the surveys
\cite{CaNoStVe:14c,CaVe:13,NoSiVe:09,Stevenson:09}.
In contrast, the state of the art for
$hp$-type finite elements is still in evolution; see
\cite{DoerflerHeuveline:2007, DoerflerBuerg:2011} and the more recent
paper \cite{CaNoStVe:15} which includes optimality estimates.

For all these cases, convergence is proven to be \emph{linear}, i.e., a certain expression controlling the error (a norm, or a combination of norm and estimator) contracts with some fixed parameter $\rho<1$ from one iteration to the next one, e.g., $\Vert u - u_{k+1} \Vert \leq \rho  \Vert u - u_{k} \Vert $. This is typically achieved if the adaptation strategy is based on some form of \emph{D\"orfler marking} (or \emph{bulk chasing}) with fixed parameter $\theta<1$: assuming that $\sum_{i \in {\cal I}} \eta_i^2$ is some additive error estimator at iteration $k$, one identifies a minimal subset ${\cal I}' \subset {\cal I}$ such that
$$
 \sum_{i \in {\cal I}'} \eta_i^2 \geq \theta^2 \sum_{i \in {\cal I}} \eta_i^2
$$ 
and utilizes ${\cal I}'$ for the construction of the new
discretization at iteration $k+1$. 
For wavelet or $h$-type fem discretizations, optimality is guaranteed by performing cautious successive adaptations, i.e., by choosing a moderate value of
$\theta$, say $0 < \theta \leq \theta_{\max } <1$ \cite{Stevenson:07}.
This avoids the need of cleaning-up the discrete solution from time to
time, by subjecting it to a \emph{coarsening} stage. 

On the
other hand, the resulting contraction factor $\rho=\rho(\theta)$ turns
out to be bounded from below by a positive constant, say $0 <
\rho_{\min } \leq \rho<1$ (related to the `condition number' of the
exact problem), regardless of the choice of $\theta$.
This entails a limitation on the speed of
convergence for infinite-order methods \cite{CaNoVe:14a,CaNoVe:14b},
but is not restrictive for fixed-order methods \cite{Stevenson:07,CaKrNoSi:08}.

It has been shown in \cite{CaNoVe:14a} that such an obstruction can be
avoided if a specific property of the differential operator
holds, namely the so-called \emph{quasi-sparsity} of the inverse of the associated stiffness matrix. Upon exploiting this information, a more aggressive marking strategy can be adopted, which judiciously enlarges the set ${\cal I}'$ coming out of D\"orfler's stage. The resulting contraction factor $\rho$ can be now made arbitrarily close to $0$ by choosing $\theta$ arbitrarily close to $1$.
 
When a method of spectral type is used, one expects a fast (possibly,
exponentially fast) decay of the discretization error for smooth
solutions. In such a situation, a slow convergence of the iterations
of the adaptive algorithm would endanger the overall performance of the
method; from this perspective, it is useful to be able to make the contraction factor as close to 0 as desired. Yet, linear convergence of the adaptive iterations is not enough to guarantee the optimality of the method. Let us explain why this occurs, and why a \emph{super-linear} convergence is preferable, using the following idealized setting. 

As customary in Nonlinear Approximation, we consider the
\emph{best $N$-term approximation error} $E_N(u)$ of the exact
solution $u$, in a suitable norm, using combinations of at most $N$
functions taken from a chosen basis. We prescribe a decay law of
$E_N(u)$ as $N$ increases, which classically for
fixed-order approximations is {\it algebraic} and reads
\begin{equation}\label{alg-decay}
\sup_N N^s E_N(u) <\infty,
\end{equation}
for some positive $s$. 
However, for infinite-order methods such as spectral approximations an
{\it exponential} law is relevant that reads
\begin{equation}\label{exp-decay}
\sup_N {\rm e}^{\eta N^\alpha} E_N(u) <\infty
\end{equation}
for some $\eta>0$ and $\alpha \in (0,1]$, where $\alpha<1$ accommodates the inclusion of $C^\infty$-functions that are not analytic.
This defines corresponding algebraic and exponential
\emph{sparsity classes} for the exact solution $u$.
These classes are related to Besov and Gevrey regularity of $u$ respectively.

We now assume the
ideal situation that at each iteration of our adaptive algorithm
\footnote{
  Throughout the paper, we write $A_k \lesssim B_k$ to indicate that
  $A_k$ can be bounded by a multiple of $B_k$, independently of the iteration counter $k$
  and other parameters which $A_k$ and $B_k$ may depend on; $A_k\eqsim B_k$
  means $A_k\lesssim B_k$ and $B_k \lesssim A_k$. }
\begin{equation}\label{ideal-decay}
\Vert u - u_{k} \Vert \eqsim   N_k^{-s}
\qquad\textrm{or}\qquad
\Vert u - u_{k} \Vert \eqsim  e^{-\eta N_k^\alpha},
\end{equation}
where $N_k$ is the cardinality of
the discrete solution $u_k$, i.e., the dimension of the
approximation space activated at iteration $k$.  We assume
in addition that the error decays
linearly from one iteration to the next, i.e., it satisfies precisely
\begin{equation}\label{ideal-contraction}
\Vert u - u_{k+1} \Vert = \rho \,  \Vert u - u_{k} \Vert.
\end{equation}

If $u$ belongs to a sparsity class of algebraic type, then one easily
gets $N_k \eqsim \rho^{-k/s}$, i.e., cardinalities grow exponentially fast and
\[
\Delta N_k := N_{k+1}-N_k \eqsim N_k
\eqsim \|u-u_k\|^{-1/s},
\]
i.e., the increment of cardinality between consecutive iterations is
proportional to the current cardinality as well as to the error raised to the power $-1/s$. The important message stemming from this ideal setting is
that for a {\em practical} adaptive algorithm 
one should be able to derive the estimates
$\Vert u - u_{k+1} \Vert \leq  \rho \,  \Vert u - u_{k} \Vert$ and
$\Delta N_k \lesssim \|u-u_k\|^{-1/s}$, because they yield
\[
N_n=\sum_{k=0}^{n-1}\Delta N_k \lesssim \sum_{k=0}^{n-1}\|u-u_k\|^{-1/s}
\le \|u-u_n\|^{-1/s} \sum_{k=0}^{n-1} \rho^{(n-k)/s}
\lesssim \|u-u_n\|^{-1/s}.
\]
This geometric-series argument is precisely the strategy used in
\cite{Stevenson:07, CaKrNoSi:08} and gives an estimate similar to \eqref{alg-decay}.
The performance of a practical adaptive algorithm is thus quasi-optimal.

If $u$ belongs to a sparsity class of exponential type, instead, the
situation changes radically. In fact, assuming
\eqref{ideal-decay} and \eqref{ideal-contraction}, one has $e^{-\eta N_k^\alpha} \eqsim \rho^k$, and so
\[
\lim_{k \rightarrow \infty} k^{-1/\alpha} N_k=
\big(|\log \rho|/\eta\big)^{1/\alpha},
\]
i.e., the cardinality $N_k$ grows polynomially.
For a practical adaptive algorithm, proving such a growth
is very hard if not impossible. This obstruction has
motivated the insertion of a coarsening stage in the adaptive
algorithm 
presented in \cite{CaNoVe:14a}.
Coarsening removes the negligible components of the discrete solution
possibly activated by the marking strategy and guarantees that the
final cardinality is nearly optimal \cite{CDDV:98,CaNoVe:14a},
but it does not account for the workload to create $u_k$.

One of the key points of the present contribution is the observation
that if the convergence of the adaptive algorithm is super-linear,
then one is back to the
simpler case of exponential growth of cardinalities which
is ameanable to a sharper performance analysis. To see this,
let us assume a super-linear
relation between consecutive errors:
\begin{equation}\label{ideal-superlinear}
\Vert u - u_{k+1} \Vert =  \Vert u - u_{k} \Vert^q
\end{equation}
for some  $q>1$. If additionally
$u_k$ satisfies \eqref{ideal-decay}, then one infers that
${\rm e}^{-\eta N_{k+1}^\alpha} \eqsim {\rm e}^{-\eta q N_{k}^\alpha}$, whence
\[
\lim_{k \rightarrow \infty} \frac{\Delta N_k}{N_k} =q^{1/\alpha}-1,
\qquad
\lim_{k \rightarrow \infty} \frac{|\log\|u-u_k\||^{1/\alpha}}{N_k} =\eta^{1/\alpha},
\]
the latter being just a consequence of \eqref{ideal-decay}.
This suggests that the geometric-series argument may be invoked again in the
optimality analysis of the adaptive algorithm.

This ideal setting does not apply directly to our
{\em practical} adaptive algorithm. We will be able to prove
estimates that are consistent with
the preceding derivation to some extend, namely
\[
\Vert u - u_{k+1} \Vert \leq  \Vert u - u_{k} \Vert^q,
\qquad \Delta N_k\le Q |\log\|u-u_k\||^{1/\bar{\alpha}},
\]
with constants $Q>0$ and $\bar{\alpha} \in (0,\alpha]$.
Invoking $\|u-u_n\| \leq \|u-u_k\|^{q^{n-k}}$, we then realize that
$$
N_n =
\sum_{k=0}^{n-1} \Delta N_k \le
Q \sum_{k=0}^{n-1} 
\big|\log\|u-u_k\| \, \big|^{1/\bar{\alpha}}
\leq 
\frac{Q q^{1/\bar{\alpha}}}{q^{1/\bar{\alpha}}-1} 
\big|\log\|u-u_n\|\, \big|^{1/\bar{\alpha}}.
$$
Setting $\bar{\eta}:=\big(\frac{Q
  q^{1/\bar{\alpha}}}{q^{1/\bar{\alpha}}-1}\big)^{-\bar{\alpha}}$,
we deduce the estimate
\[
\sup_n {\rm e}^{\bar{\eta} N_n^{\bar{\alpha}}} \|u-u_n\| \leq 1,
\]
which is similar to \eqref{exp-decay}, albeit with different class parameters.
The most important parameter is $\bar{\alpha}$. Its possible
degradation relative to $\alpha$
is mainly caused by the fact that the residual, the only
computable quantity accessible to our practical algorithm, belongs to
a sparsity class with a main parameter generally smaller than
that of the solution $u$.
This perhaps unexpected property is
typical of the exponential class and has
been elucidated in \cite{CaNoVe:14a}.
  
In order for the marking strategy to guarantee super-linear
convergence, one needs to adopt a dynamic choice of D\"orfler's
parameter $\theta$, which pushes its value towards $1$ as the
iterations proceed. We accomplish this requirement by equating the
quantity $1-\theta_k^2$ to some function of the dual norm of the
residual $r_k$, which is monotonically increasing and
vanishing at the origin.
This defines our \emph{dynamic marking strategy}.  The order of the
root at the origin dictates the exponent $q$ in the super-linear
convergence estimate of our adaptive algorithm.

The paper is organized as follows. In Sect. 2 we introduce the model
elliptic problem and its spectral Galerkin approximation based on
either multi-dimensional Fourier or (modified) Legendre expansions.
In particular, we highlight properties of the resulting
stiffness matrix that will be fundamental
in the sequel. We present the adaptive algorithm in
Sect. 3, first for the static marking ($\theta$ fixed) and
later for the dynamic marking ($\theta$ tending towards 1); super-linear
convergence is proven. With the optimality analysis in mind, we next
recall in Sect. 4 the definition and crucial properties of a family of
sparsity classes of exponential type, related to Gevrey
regularity of the solution, and we investigate how the sparsity class
of the Galerkin residual deteriorates relative to that of the exact
solution. Finally, in Sect. 5 we relate the cardinality of the
adaptive discrete solutions, as well as the workload needed to compute
them, to the expected accuracy of the approximation. Our analysis
confirms that the proposed dynamic marking strategy avoids any form of
coarsening, while providing exponential convergence
with linear computational complexity, assuming optimal linear solvers. 

\section{Model Elliptic Problem and Galerkin Methods}

Let $d \geq 1$ and consider the following elliptic PDE in a $d$-dimensional rectangular domain $\Omega$
with periodic  or homogeneous Dirichlet boundary conditions:
\begin{equation}\label{eq:four03}
{ L}u=-\nabla \cdot (\nu \nabla u)+ \sigma u = f \qquad \text{in } \Omega ,
\end{equation}
where $\nu$ and $\sigma$ are sufficiently smooth real coefficients satisfying 
$0 < \nu_* \leq \nu(x) \leq \nu^* < \infty$ and $0 < \sigma_* \leq \sigma(x) \leq \sigma^* < \infty$
in $\Omega$; let us set
$$
\alpha_* = \min(\nu_*, \sigma_*) \qquad \text{and} \qquad \alpha^* = \max(\nu^*, \sigma^*) \;.
$$
Let $V$ be equal to $H^1_0(\Omega)$ or $H^1_p(\Omega)$ depending on the boundary conditions and denote by $V^*$ its dual space. We formulate \eqref{eq:four03} variationally as
\begin{equation}\label{weak}
u \in V \ \ : \quad a(u,v)= \langle f,v \rangle \qquad \forall v \in  V \;,
\end{equation}
where $a(u,v)=\int_\Omega \nu \nabla u \cdot \nabla \bar{v} + \int_\Omega \sigma u \bar{v}$ (bar 
indicating as usual complex conjugate). We denote by 
$\tvert v \tvert = \sqrt{a(v,v)}$
the energy norm of any $v \in V$, which satisfies 
\begin{equation}\label{eq:four.1bis}
\sqrt{\alpha_*}  \Vert v \Vert_V  \leq \tvert v \tvert \leq 
\sqrt{\alpha^*}  \Vert v \Vert_V \;.
\end{equation}
\subsection{Riesz Basis}
We start with an abstract formulation which encompasses the two
examples of interest: trigonometric functions and Legendre polynomials.
Let $\phi=\{\phi_k \, : \, k \in {\cal K}\}$ be a Riesz basis of $V$. Thus, we assume 
the following relation between a function
$v = \sum_{k \in {\cal K}} \hat{v}_k \phi_k\in V$ and its coefficients:
\begin{equation}\label{eq:propNOBS.3}
\Vert v \Vert_{V}^2  \simeq \sum_{k \in {\cal K}}
|\hat{v}_k|^2 d_k=:\Vert v \Vert_\phi^2 \; ,
\end{equation}
for suitable weights $d_k>0$.
{
Correspondingly, any element $f \in V^*$ can be expanded along the 
{\sl dual basis} $\phi^*=\{\phi_k^*\}$ as 
$f = \sum_{k \in {\cal K}} \hat{f}_k \phi_k^*$, with $\hat{f}_k = \langle f,\phi_k \rangle$,
yielding the dual norm representation
\begin{equation}\label{eq:propNOBS.4}
\Vert f \Vert_{V^*}^2  \ \simeq \  \sum_{k \in {\cal K}} |\hat{f}_k|^2 d_k^{-1}=:\Vert v \Vert_{\phi*}^2\;.
\end{equation}
For future reference, we introduce the vectors $\bv = (\hat v_k d_k^{1/2})_{k\in {\cal K}}$
and $\bF = (\hat f_k d_k^{-1/2})_{k\in {\cal K}}$ as well as the
constants $\beta_*\le\beta^*$ of the norm equivalence in \eqref{eq:propNOBS.3}
\begin{equation}\label{eq:propNOBS.7}
\beta_* \Vert v \Vert_{V} \leq \Vert v \Vert_\phi = \|\bv\|_{\ell^2} \leq \beta^* \Vert v \Vert_{V}
\qquad \forall v \in V \;.
\end{equation}
This implies
\begin{equation}\label{eq:propNOBS.8}
\frac1{\beta^*} \Vert f \Vert_{V^*} \leq \Vert f \Vert_{\phi^*} =
\|\bF\|_{\ell^2} \leq \frac1{\beta_*} \Vert f \Vert_{V^*}
\qquad \forall f \in V^* \;.
\end{equation}
}

The two key examples
to keep in mind are trigonometric basis and tensor products of
Babu\v ska-Shen basis. We discuss them briefly below.

\medskip
{\bf Trigonometric basis}.
Let $\Omega=(0,2\pi)^d$ and the trigonometric basis be
$
\phi_k(x)=\frac1{(2\pi)^{d/2}} \, {\rm e}^{i k \cdot x}
$
for any $k \in \mathcal{K}=\mathbb{Z}^d$ and $x \in \Omega$. 
Any function $v\in L^2(\Omega)$ can be expanded in terms of
$\{\phi_k\}_{k\in\mathbb{Z}^d}$ as follows:
\begin{equation}\label{fourier-exp}
v = \sum_k \hat{v}_k \phi_k \;, \qquad \hat{v}_k=\langle v,\phi_k\rangle \;, 
\qquad \ \Vert v \Vert_{L^2(\Omega)}^2= \sum_k |\hat{v}_k|^2 \;.
\end{equation}
The space $V := H^1_p(\Omega)$ of periodic functions with square
integrable weak gradient
can now be easily characterized as the subspace of those $v \in L^2(\Omega)$ for which
$$
\Vert v \Vert_V^2 =
\Vert v \Vert_{H^1_p(\Omega)}^2 = \sum_k |\hat{V}_k|^2 <\infty 
\qquad (\text{where }\hat{V}_k := \hat{v}_k d_k^{1/2}, \text{ with }{d_k}:=1+|k|^2).
$$
This induces an {\it isomorphism} between $H^1_p(\Omega)$ and $\ell^2(\mathbb{Z}^d)$:
for each $v \in H^1_p(\Omega)$ let $\bv=(\hat{V}_k)_{k \in {\cal K}} \in\ell^2(\mathbb{Z}^d)$ and
note that $\|v\|_{H^1_p(\Omega)} = \|\bv\|_{\ell^2}$.
Likewise, the dual space
$H^{-1}_p(\Omega)=(H^1_p(\Omega))'$ is characterized as the space of those functionals $f$ for which 
$$
\Vert f \Vert_{V^*}^2 =
\Vert f \Vert_{H^{-1}_p(\Omega)}^2 = \sum_k |\hat{F}_k|^2
\qquad\text{with}\quad
\hat{F}_k := \hat{f}_k d_k^{-1/2}.
$$
We also have an isomorphism between $H^{-1}_p(\Omega)$ and $\ell^2(\mathbb{Z}^d)$
upon setting $\bF=(\hat{F}_k)_{k \in {\cal K}}$ for $f\in H^{-1}_p(\Omega)$ and
realizing that $\|f\|_{H^{-1}_p(\Omega)} = \|\bF\|_{\ell^2}$.

\medskip
{\bf Babu\v ska-Shen basis}.
Let us start with the one-dimensional case $d=1$. Set $I=(-1,1)$, $V:=H^1_0(I)$, and let 
$L_k({x})$, $k \geq 0$, stand for the $k$-th Legendre orthogonal polynomial in $I$, 
which satisfies  ${\rm deg}\, L_k = k$, $L_k(1)=1$ and
\begin{equation}\label{eq:Leg-ort}
\int_I L_k({x}) L_m({x}) \, d{x} = \frac2{2k+1}\, \delta_{km}\;, \qquad m \geq 0 \;.
\end{equation}
The natural modal basis in $H^1_0(I)$ is the {\sl Babu\v ska-Shen basis} (BS basis), whose elements are defined as
\begin{equation}\label{eq:defBS}
\eta_k({x})=\sqrt{\frac{2k-1}2}\int_{{x}}^1 L_{k-1}(s)\,{d}s =
\frac1{\sqrt{4k-2}}\big(L_{k-2}({x})-L_{k}({x})\big)\ , 
 \qquad k \geq 2\;.
\end{equation}
The basis elements satisfy ${\rm deg}\, \eta_k = k$ and
\begin{equation}\label{eq:propBS.1}
(\eta_k,\eta_m)_{H^1_0({I})} = 
\int_I \eta_k'({x}) \eta_m'({x}) \, d{x} = \delta_{km}\;, \qquad k,m \geq 2 \;,
\end{equation}
i.e., they form an orthonormal basis for the ${H^1_0({I})}$-inner product. Equivalently, 
the (semi-infinite) stiffness matrix ${S}_\eta$ of the Babu\v ska-Shen basis with respect to this inner product is the identity matrix. 

We now consider, for simplicity, the two-dimensional case $d=2$ since
the case $d>2$ is similar. Let $\Omega=(-1,1)^2$, $V=H^1_0(\Omega)$,
and consider the {\sl tensorized Babu\v ska-Shen basis}, 
whose elements are defined as
\begin{equation}\label{eq:defBS.2}
\eta_k(x) = \eta_{k_1}(x_1) \eta_{k_2}(x_2)\;, \qquad k_1, k_2 \geq 2 \;, 
\end{equation}
where we set $k=(k_1,k_2)$ and $x=(x_1, x_2)$; indices vary in the set
${\cal K}=\{k \in \mathbb{N}^2 \, : \, k_i \geq 2 \text{ for }i=1,2 \}$,
which is ordered `a la Cantor' by increasing total degree $k_\text{tot}=k_1+k_2$ and, for the same total degree, by increasing $k_1$. 
The tensorized BS basis is no longer orthogonal, since
\begin{equation}\label{eq:orthogonal}
(\eta_k,\eta_m)_{H^1_0(\Omega)} = (\eta_{k_1},\eta_{m_1})_{H^1_0({I})}(\eta_{k_2},\eta_{m_2})_{L^2({I})}+
 (\eta_{k_1},\eta_{m_1})_{L^2({I})}(\eta_{k_2},\eta_{m_2})_{H^1_0({I})} \;,
\end{equation}
whence $ (\eta_k,\eta_m)_{H^1_0(\Omega)} \not = 0$ if and only if $k_1=m_1$ and $k_2-m_2 \in \{-2,0,2\}$, or
$k_2=m_2$ and $k_1-m_1 \in \{-2,0,2\}$. Obviously, we cannot have a Parseval representation of the  $H^1_0(\Omega)$-norm of $v = \sum_{k \in {\cal K}} \hat{v}_k \eta_k$ in terms of the coefficients  $\hat{v}_k$. 
With the aim of getting \eqref{eq:propNOBS.3}, we follow
\cite{CaSiVe:14-2} and we first perform the orthonormalization of the BS
basis via a Gram-Schmidt procedure. This allows us to build a sequence of functions
\begin{equation}\label{eq:defOBS}
\Phi_k = \sum_{m \leq k}  g_{mk} \eta_m \;,
\end{equation}
such that  $g_{kk}\not = 0$ and 
$$
(\Phi_k,\Phi_m)_{H^1_0(\Omega)} = \delta_{km} \qquad \forall \ k, m \in {\cal K}\;.
$$
We will refer to the collection $\Phi:=\{\Phi_k:\ k \in\mathcal{K}\}$
as the {\sl orthonormal Babu\v ska-Shen basis} (OBS basis), for which
the associated stiffness matrix ${S}_\Phi$ with respect to the $H^1_0(\Omega)$-inner product is the identity matrix.
Equivalently, if ${G}=(g_{mk})$ is the upper triangular matrix which collects the coefficients generated by the
Gram-Schmidt algorithm above, one has
\begin{equation}\label{eq:propOBS.1}
{G}^T {S}_\eta {G} = {S}_\Phi = {I} \;,
\end{equation}
that is the validity of  \eqref{eq:propNOBS.7} with $d_k=1$.
However, unlike ${S}_\eta$, which is very sparse, the upper
triangular matrix ${G}$ is full; in view of this, we next apply a
thresholding procedure to wipe-out a significant portion of the non-zero entries sitting in the leftmost columns of ${G}$.
This  leads to a modified basis whose computational efficiency is quantitatively improved,
without significantly deteriorating the properties of the OBS basis. To be more precise, we
use the following notation: $G_t$ indicates the matrix obtained from
$G$ by setting to zero a certain finite set of off-diagonal entries, so that in particular
${\rm diag}(G_t)= {\rm diag}(G)$; correspondingly, $E:=G_t-G$
is the matrix measuring the truncation quality, for which ${\rm diag}(E)=0$. 
Finally, we introduce the matrix 
\begin{equation}\label{eq:defNOBS:1}
{S}_\phi = {G}^T_t {S}_\eta {G}_t 
\end{equation}
which we interpret as the stiffness matrix associated to the modified BS basis defined in analogy to \eqref{eq:defOBS} as 
\begin{equation}\label{eq:defNOBS}
\phi_k = \sum_{m \in \mathcal{M}_t(k)}  g_{mk} \eta_m \;
\end{equation} 
where $\mathcal{M}_t(k)=\{m \leq k: E_{mk}=0\}$. This forms a new basis in
$H^1_0(\Omega)$ (because $k \in \mathcal{M}_t(k)$ and $g_{kk}\not =0$).
We will term it a {\sl nearly-orthonormal Babu\v ska-Shen basis} (NOBS basis).
Note that only the basis functions $\phi_k$ having total degree not exceeding a certain value, say $p$, may be affected by  the compression, while all the others  coincide with the corresponding orthonormal basis functions $\Phi_k$ defined in \eqref{eq:defOBS}. 

If $D_{\phi} = {\rm diag}\,S_\phi$, then for any value of $p$ there are strategies to build $G_t$ (depending on $p$)such that  the eigenvalues $\lambda$ of
\begin{eqnarray}\label{eqn:maineigpb}
S_\phi x = \lambda D_{\phi} x
\end{eqnarray} are close to one and bounded from above and away from 0, independently of $p$ \cite{CaSiVe:14-2}. 
This guarantees the validity for NOBS basis of \eqref{eq:propNOBS.3} with $d_k$ equal to the diagonal elements of the matrix $D_\phi$ and \eqref{eq:propNOBS.7} for suitable choice of constants $\beta_*,\beta^*$ depending on the eigenvalues of \eqref{eqn:maineigpb} (see \cite{CaSiVe:14-2} for more details).

\subsection{Infinite Dimensional Algebraic Problem}\label{S:properties-A}
Let us identify the solution $u = \sum_k \hat{u}_k \phi_k$ of Problem (\ref{eq:four03})
with the vector $\bu=(\hat{u}_k)_{k\in {\cal K} }$ 
of its coefficients w.r.t. the basis $\{\phi_k\}_{k\in\mathcal{K}}$. Similarly, let us identify
the right-hand side $f$ with the vector $\bF=(\hat{f}_\ell)_{\ell \in {\cal K}}$ of its dual coefficients.
Finally, let us introduce the bi-infinite, symmetric and positive-definite stiffness matrix 
\begin{equation}\label{eq:four100}
{{\bf A}}=(a_{\ell, k})_{\ell,k \in {\cal K}} \qquad \text{with} \qquad 
a_{\ell, k}= a(\phi_k,\phi_\ell)\;.
\end{equation}
Then, Problem (\ref{eq:four03}) can be equivalently written as
\begin{equation}\label{eq:four110}
{{\bf A}} \bu = \bF \;,
\end{equation}
where, thanks to  \eqref{eq:four.1bis} and the norm equivalences \eqref{eq:propNOBS.7}-\eqref{eq:propNOBS.8}, ${\bf A}$ defines a bounded invertible operator 
in  {$\ell^2({\cal{K}})$}. 

\medskip\noindent 
{\bf Decay Properties of $\bA$ and $\bA^{-1}$}.
%
The decay of the entries of $\bA$ away from the diagonal depends on the
regularity of the coefficients $\nu$ and $\sigma$ of $L$.
%
If $\nu$ and $\sigma$ are real analytic in a neighborhood of $\Omega$, then
$a_{k,m}$ decays exponentially away from the diagonal \cite{CaNoVe:14a,CaNoVe:14b,CaSiVe:14-2}:
there exist parameters $c_L,\eta_L>0$ such that
\begin{equation}\label{eq:decay-entries-A}
|a_{k,m}| \leq \ c_L {\rm exp}(-\eta_L |k-m|) \quad\forall k,m \in \mathcal{K};
\end{equation}
we then say that $\bA$ belongs to the exponential class
$\cD_e(\eta_L)$, in particular $\bA$ is quasi-sparse.
This justifies the {\it symmetric truncation} $\bA_J$ of $\bA$ with {\it parameter $J$}, defined as
$(\mathbf{A}_J)_{\ell,k}=a_{\ell,k}$ if $|\ell-k| \leq J$
and $(\mathbf{A}_J)_{\ell,k}=0$ otherwise, which satisfies \cite{CaNoVe:14a,CaNoVe:14b,CaSiVe:14-2}
\begin{equation}\label{decay-A}
\Vert \mathbf{A}-\mathbf{A}_J \Vert \leq 
C_{\mathbf{A}} (J+1)^{d-1}{\rm e}^{-\eta_L J} 
\end{equation}
for some $C_{\mathbf{A}}>0$ depending only on $c_L$.

Most notably, the inverse matrix $\mathbf{A}^{-1}$ is also
quasi-sparse \cite{CaNoVe:14a,CaNoVe:14b,CaSiVe:14-2}. Precisely  $\bA^{-1} \in \cD_e(\bar\eta_L)$ for some $\bar{\eta}_L \in (0,\eta_L]$ and $\bar{c}_{L}$  only dependent on $c_L$ and $\eta_L$. 
Thus, there exists an explicit constant $C_{\mathbf{A}^{-1}}$ (depending only on $c_L$ and $\eta_L$) such that the symmetric truncation $(\mathbf{A}^{-1})_J$ of $\mathbf{A}^{-1}$ satisfies 

\begin{equation}\label{decayA-1}
\Vert \mathbf{A}^{-1}-(\mathbf{A}^{-1})_J \Vert \leq 
C_{\mathbf{A}^{-1}} (J+1)^{d-1} {\rm e}^{-\bar{\eta}_L J} \leq
C_{\mathbf{A}^{-1}} {\rm e}^{-\tilde{\eta}_L J} \; ,
\end{equation}
for a suitable exponent $\wt\eta_L < \bar\eta_L$.

\medskip\noindent
{\bf Galerkin Method.}
{
Given any finite index set $\Lambda \subset {\cal K}$, we define the subspace $V_{\Lambda} = {\rm span}\,\{\phi_k\, | \, k \in \Lambda \}$ of $V$;  we set $|\Lambda|= \rm{card}\, \Lambda$, so that $\rm{dim}\, V_{\Lambda}=|\Lambda|$. If $v\in V$ admits the
expansion $v = \sum_{k \in {\cal K}} \hat{v}_k \phi_k $, then we define its 
projection $P_\Lambda v$ upon $V_\Lambda$ by setting
$P_\Lambda v := \sum_{k \in \Lambda} \hat{v}_k \phi_k$.
Similarly, we define the subspace $V_{\Lambda}^* = {\rm
  span}\,\{\phi^*_k\, | \, k \in \Lambda \}$ of $V^*$. If $f$ admits an expansion $f = \sum_{k \in {\cal K}} \hat{f}_k \phi^*_k $, then we define its  projection $P^*_\Lambda f$ onto $V^*_\Lambda$ upon setting
$P^*_\Lambda f := \sum_{k \in \Lambda} \hat{f}_k \phi^*_k $.

\smallskip
Given any finite $\Lambda \subset {\cal K}$, the Galerkin approximation of \eqref{eq:four03} is defined as
\begin{equation}\label{eq:four.2}
u_\Lambda \in V_\Lambda \ \ : \quad a(u_\Lambda,v_\Lambda)= 
\langle f,v_\Lambda \rangle \qquad \forall v_\Lambda \in V_\Lambda \;.
\end{equation}

Let  $\mathbf{u}_\Lambda$ be the vector collecting the coefficients
of $u_\Lambda$ indexed in $\Lambda$; 
let $\mathbf{f}_\Lambda$ be the analogous restriction for
the vector of the coefficients of $f$. Finally, denote by $\mathbf{R}_\Lambda$ the
matrix that restricts a vector indexed in ${\cal K}$ to the portion indexed in $\Lambda$, so that
$\mathbf{R}_\Lambda^H$ is the corresponding extension matrix.
If
\begin{equation}\label{eq:four120}
\mathbf{A}_\Lambda := \mathbf{R}_\Lambda \mathbf{A} \mathbf{R}_\Lambda^H \;,
\end{equation}
then problem (\ref{eq:four.2}) can be equivalently written as
\begin{equation}\label{eq:four130}
\mathbf{A}_\Lambda \mathbf{u}_\Lambda = \mathbf{f}_\Lambda \;. 
\end{equation}

For any $w \in V_\Lambda$, we define the residual $r(w) \in V^*$ as
$$
r(w)=f-{L}w = \sum_{k \in {\cal K}} \hat{r}_k(w) \phi^*_k \;, \qquad \text{where} \qquad 
\hat{r}_k(w) = \langle f - {L}w, \phi_k \rangle = \langle f,\phi_k \rangle -a(w,\phi_k) \;.
$$
The definition \eqref{eq:four.2} of $u_\Lambda$ is equivalent  to the condition $P^*_\Lambda r(u_\Lambda) = 0$, i.e., $\hat{r}_k(u_\Lambda)=0$ for every $k \in \Lambda$.
By the continuity and coercivity of the bilinear form, one has 
\begin{equation}\label{eq:four.2.1}
\frac1{\alpha^*} \Vert r(u_\Lambda) \Vert_{V^*} \leq
\Vert u - u_\Lambda \Vert_{V} \leq 
\frac1{\alpha_*} \Vert r(u_\Lambda) \Vert_{V^*} \;,
\end{equation}
which in view of \eqref{eq:four.1bis} and \eqref{eq:propNOBS.8} can be rephrased as 
\begin{equation}\label{eq:four.2.1bis}
\frac{\beta_*}{\sqrt{{\alpha^*}}} \Vert r(u_\Lambda) \Vert_{\phi^*} \leq
\tvert u - u_\Lambda \tvert  \leq 
\frac {\beta^*}{\sqrt{\alpha_*}} \Vert r(u_\Lambda) \Vert_{\phi^*} \;.
\end{equation}
}
Therefore, if $(\hat{r}_k(u_\Lambda))_{k \in {\cal K}}$ are the coefficients of ${r}(u_\Lambda)$ with respect to the dual basis $\phi^*$, the quantity
\begin{equation*}\label{eq:four.2bis}
\Vert r(u_\Lambda) \Vert_{\phi^*}
=\left( \sum_{k \not \in \Lambda} |\hat{R}_k(u_\Lambda)|^2 \right)^{1/2} \qquad\text{with}\quad \hat{R}_k(u_\Lambda) =  \hat{r}_k(u_\Lambda) d_k^{-1/2}
\end{equation*}
is an {\it error estimator} from above and from below.
However, this quantity is not computable because it involves infinitely
many terms. We discuss {\it feasible versions} in
\cite{CaNoVe:14a,CaNoVe:14b, CaSiVe:14-2} but not here.
\medskip 

\noindent {\bf Equivalent Formulation of the Galerkin Problem.}
For future reference, we now rewrite the Galerkin problem \eqref{eq:four.2} in an equivalent (infinite-dimensional) manner.
Let  $$\mathbf{P}_\Lambda: \ell^2(\mathcal{K}) \to \ell^2(\mathcal{K})$$ be the projector operator defined as 
\[
(\mathbf{P}_\Lambda \mathbf{v})_\lambda :=
\begin{cases}
v_\lambda & \text{\rm if } \lambda\in\Lambda \;, \\
0 & \text{\rm if } \lambda\notin\Lambda \;.
\end{cases}
\]
Note that $\mathbf{P}_\Lambda$ can be represented  as a diagonal bi-infinite matrix whose diagonal elements 
are $1$ for indexes belonging to $\Lambda$, and zero otherwise.  
We set $\mathbf{Q}_\Lambda := \mathbf{I}-\mathbf{P}_\Lambda$ and 
introduce the bi-infinite matrix $\widehat{\mathbf{A}}_\Lambda:=
\mathbf{P}_\Lambda \mathbf{A} \mathbf{P}_\Lambda + \mathbf{Q}_\Lambda$ which 
is equal to $\mathbf{A}_\Lambda$ for indexes in $\Lambda$ and to the identity matrix, otherwise. 
The definitions of the projectors $\mathbf{P}_\Lambda$ and
$\mathbf{Q}_\Lambda$ yield the following property:
\begin{equation}\label{prop:inf-matrix}
\textit{If $\mathbf{A}$ is invertible with
 $\mathbf{A}\in\mathcal{D}_e(\eta_L)$, then the same holds for
$\widehat{\mathbf{A}}_\Lambda$.}
\end{equation}
Furthermore, the constants $C_{\widehat{\mathbf{A}}_\Lambda}$ and $C_{(\widehat{\mathbf{A}}_\Lambda)^{-1}}$ which appear in the inequalities \eqref{decay-A} and \eqref{decayA-1} for $\widehat{\mathbf{A}}_\Lambda$ can be bounded uniformly in $\Lambda$, since in turn they can be bounded in terms of $\eta_L$ and $c_L$, respectively.


\noindent Now, let us consider the following extended Galerkin problem: 
find $\hat{\mathbf{u}}\in\ell^2$ such that  
\begin{equation}\label{eq:inf-pb-galerkin}
\widehat{\mathbf{A}}_\Lambda \hat{\mathbf{u}}
= \mathbf{P}_\Lambda \mathbf{f}\ .
\end{equation}
Let $\mathbf{u}_\Lambda$ be the Galerkin solution to  \eqref{eq:four130};
then, it is easy to check that $\hat{\mathbf{u}}={\mathbf{R}}_\Lambda^H \mathbf{u}_\Lambda$.

\section{Adaptive Spectral Galerkin Method}\label{S:agressive-gal}
In this section we present our adaptive spectral Galerkin method,
named {\bf DYN-GAL}, that is based on a new notion of marking
strategy, namely a {\it dynamic marking}. In Section
\ref{S:static-marking} we recall the enriched D\"orfler marking
strategy, introduced in \cite{CaNoVe:14a}, which represents an enhancement of the classic D\"orfler marking strategy. In Section \ref{S:dyn-marking} we introduce the dynamic marking strategy, present {\bf DYN-GAL} and prove its quadratic convergence.  
\subsection{Static D\"orfler Marking}\label{S:static-marking}
Fix any $\theta \in (0,1)$ and set
$\Lambda_0= \emptyset$, $u_{\Lambda_0}=0$.
For $n=0,1, \dots$, assume that $\Lambda_n$ and $u_n := u_{\Lambda_n} \in V_{\Lambda_n}$
and $r_n := r(u_n) = Lu_n-f$
are already computed and choose 
$\Lambda_{n+1} := \Lambda_n \cup \partial\Lambda_n$ where the set $\partial\Lambda_n$ is built by a two-step procedure that we call $\textbf{E-D\"ORFLER}$ for {\it enriched D\"orfler}:
\medskip
\begin{itemize}
\item[] $\partial\Lambda_n = \textbf{E-D\"ORFLER} \, (\Lambda_n,\theta)$
\begin{enumerate}
\item[] $\widetilde{\partial\Lambda}_n={\bf{DORFLER}} \, (r_n,\theta)$
\item[] ${\partial\Lambda}_n={\bf ENRICH} \, (\widetilde{\partial\Lambda}_n,J)$
\end{enumerate}
\end{itemize}
\medskip
The first step is the usual {\it D\"orfler's marking}
with parameter $\theta$:
\begin{equation}\label{doerfler}
  \Vert P_{\wt{\partial \Lambda}_{n}}^* r_n \Vert_\ps
  = \Vert P_{\wt\Lambda_{n+1}}^* r_n \Vert_\ps \geq \theta
\Vert  r_n \Vert_\ps
\quad\textrm{or}\quad
\sum_{k \in \wt{\partial\Lambda}_{n}}  |\hat{R}_k(u_n)|^2 
\ \geq \ \theta^2  \sum_{k \in \mathcal{K}}  |\hat{R}_k(u_n)|^2 \;,
\end{equation}
with $\wt\Lambda_{n+1}=\Lambda_{n}\cup \wt\partial\Lambda_{n}$.
This also reads
\begin{equation}\label{doerfler-equiv}
\Vert r_n-P_{\wt\Lambda_{n+1}}^* r_n \Vert_\ps \leq \sqrt{1-\theta^2}
\Vert r_n \Vert_\ps \;
\end{equation}
and can be implemented by rearranging the coefficients
$\hat{R}_k(u_n)$ in decreasing order of modulus
and picking the largest ones ({\it greedy approach}).
However, this is only an idealized algorithm because the number of coefficients
$\hat{R}_k(u_n)$ is infinite. This marking is known to yield a
contraction property between $u_n$  and the Galerkin solution $\wt u_{n+1}\in V_{\wt\Lambda_{n+1}}$ of the form
\[
  \tvert u-\wt u_{n+1} \tvert   \leq 
  \rho(\theta) \tvert u-u_n \tvert \;,
\]
with $\rho(\theta) = \sqrt{1-\frac{\alpha_*}{\alpha^*}\theta^2}$
\cite{CaNoVe:14a,CaNoVe:14b}. When $\alpha_* < \alpha^*$ we see that
in contrast to \eqref{doerfler-equiv}, $\rho(\theta)$ is bounded below
away from $0$ by $\sqrt{1-\frac{\alpha_*}{\alpha^*}}$.

The second step of $\textbf{E-D\"ORFLER}$ is meant to remedy this
situation and hinges on the a priori structure of $\bA^{-1}$ already
alluded to in \S \ref{S:properties-A}. The goal is to augment the set
$\wt{\partial\Lambda}_{n}$ to $\partial\Lambda_{n}$
judiciously. This is contained in the following proposition (see \cite{CaNoVe:14a}) whose proof is reported here for completeness.
\begin{proposition}[enrichment]
Let $\widetilde{\partial\Lambda}_n={\bf{DORFLER}} \, (r_n,\theta)$, and let $J=J(\theta)>0$ satisfy 
\begin{equation}\label{eq:aggr2}
C_{\bA^{-1}} \textrm{e}^{-\wt\eta_L J} \leq \sqrt{\frac{1-\theta^2}{\alpha_* \alpha^*}} \;,
\end{equation}
where $C_{\bA^{-1}}$ and $\wt\eta_L$ are defined in  \eqref{decayA-1}.
Let ${\partial\Lambda}_n={\bf ENRICH} \, (\widetilde{\partial\Lambda}_n,J)$ be built as follows
$$
\partial\Lambda_{n} := \big \{ k\in\mathcal{K}: \quad\textrm{there exists } \ell\in\wt{\partial\Lambda}_{n}
\textrm{ such that }  |k-\ell|\le J \big\}.
$$ 
Then for $\Lambda_{n+1}=\Lambda_n\cup \partial\Lambda_n$, the Galerkin solution $u_{n+1}\in V_{\Lambda_{n+1}}$ satisfies 
\begin{equation}
\tvert u - u_{n+1} \tvert \leq  \bar{\rho}(\theta)  \tvert u - u_{n} \tvert \;
\end{equation}
with  
\begin{equation}\label{eq:aggr3}
\bar{\rho}(\theta)=2
\frac{\beta^*\sqrt{\alpha^*}}{\beta_*\sqrt{\alpha_*}}\sqrt{1-\theta^2}.
\end{equation}
\end{proposition}

\begin{proof}
Let $g_n := P^*_{\wt{\partial\Lambda}_n} r_n =  P^*_{\wt\Lambda_{n+1}} r_n$ which, according to
\eqref{doerfler-equiv}, satisfies
$$
\Vert r_n- g_n \Vert_\ps \leq \sqrt{1-\theta^2}  \Vert r_n \Vert_\ps \; .
$$
Let $w_n \in V$ be the solution of $L w_n = g_n$, which in general
will have infinitely many components, and let us split it as
$$
w_n= P_{\Lambda_{n+1}} w_n + P_{\Lambda_{n+1}^c} w_n
=: y_n + z_n \in V_{\Lambda_{n+1}} \oplus 
 V_{\Lambda_{n+1}^c} \;.
$$
The minimality property  in the energy norm
of the Galerkin solution $u_{n+1}$ over the set $\Lambda_{n+1}$ yet to be defined,
in conjunction with \eqref{eq:four.1bis}
and \eqref{eq:four.2.1bis}, implies
\begin{align*}
\tvert u - u_{n+1} \tvert &\leq  \tvert u - (u_{n}+y_{n}) \tvert \leq  
 \tvert u- u_n - w_{n} + z_{n} \tvert \\
&\leq \frac1{\sqrt{\alpha_*}} \Vert L(u- u_n - w_{n}) \Vert + \sqrt{\alpha^*}\Vert z_{n} \Vert
= \frac{\beta^*}{\sqrt{\alpha_*}} \Vert r_n - g_{n} \Vert_\ps  + \sqrt{\alpha^*} \Vert z_{n} \Vert \;,
\end{align*}
whence
$$
\tvert u - u_{n+1} \tvert \leq
\frac{\beta^*}{\sqrt{\alpha_*}}\sqrt{1-\theta^2} \, \Vert r_n \Vert_\ps
+ \sqrt{\alpha^*} \Vert z_{n}\Vert \;.
$$ 
Since
$z_n= \big( P_{\Lambda_{n+1}^c} L^{-1}P^*_{\widetilde{\partial\Lambda}_n} \big) r_n $,
we now construct $\Lambda_{n+1}^c$ to control $\|z_n\|$. If
$$
 k \in \Lambda_{n+1}^c \quad \text{and} \quad \ell \in \widetilde{\partial\Lambda}_n\qquad \Rightarrow \qquad
|k - \ell | > J \;,
$$
then we have 
$$
\Vert P_{\Lambda_{n+1}^c} L^{-1} P^*_{\widetilde{\partial\Lambda}_n} \Vert \leq
\Vert \mathbf{A}^{-1}-(\mathbf{A}^{-1})_J \Vert \leq 
C_{\bA^{-1}} \textrm{e}^{-\wt\eta_L J} \;,
$$ 
where we have used \eqref{decayA-1}. We now choose $J=J(\theta)>0$ to satisfy \eqref{eq:aggr2},
and we exploit that $\|z_n\|\le C_{\bA^{-1}} e^{-\tilde\eta_L J} \|r_n\|$ to obtain
\begin{equation}\label{eq:aggr_error_reduct}
\tvert u - u_{n+1} \tvert \leq 2\frac{\beta^*}{\sqrt{\alpha_*}}
\sqrt{1-\theta^2} \, \Vert r_n \Vert_\ps
\leq 2\frac{\beta^*\sqrt{\alpha^*}}{\beta_*\sqrt{\alpha_*}}
\sqrt{1-\theta^2} \, \tvert u - u_{n} \tvert \; ,
\end{equation}
as asserted.
\end{proof}

We observe that, as desired, the new error reduction rate
\begin{equation}
\bar{\rho}(\theta)=2
\frac{\beta^*\sqrt{\alpha^*}}{\beta_*\sqrt{\alpha_*}}\sqrt{1-\theta^2}
\end{equation}
can be made arbitrarily small by choosing $\theta$ suitably close to $1$.
This observation was already made in \cite{CaNoVe:14a,CaNoVe:14b}, but
we improve it in Section \ref{S:dyn-marking} upon choosing $\theta$ dynamically.

\begin{remark}[Cardinality of $\partial\Lambda_n$] Since we add a ball of radius $J$ around each point of
$\wt{\partial\Lambda}_{n}$ we get a crude estimate
\begin{equation}\label{card-Lambda}
|\partial\Lambda_{n}| \le |B_d(0,J) \cap \Z^d| ~ |\widetilde{\partial\Lambda}_{n}| \approx
\omega_d J^d  |\widetilde{\partial\Lambda}_{n}|,
\end{equation}
where $\omega_d$ is the measure of the $d$-dimensional Euclidean unit
ball $B(0,1)$ centered at the origin.
\end{remark}
\subsection{Dynamic D\"orfler Marking and Adaptive Spectral Algorithm}\label{S:dyn-marking}
In this section we improve on the above marking strategy upon making the choice of $\theta$ dynamic.
At each iteration $n$ let us select the D\"orfler parameter $\theta_n$ such that 
\begin{equation}\label{aux:1}
\sqrt{1-\theta_n^2}=C_0\frac{\| r_n \|_\ps}{\| r_0 \|_\ps}
\end{equation}
for a proper choice of the positive constant $C_0$ that will be made precise later.
This implies
\begin{equation}\label{eq:Jk}
J(\theta_n) = - \frac{1}{\wt{\eta}_L} \log \frac{\| r_n \|_\ps}{\| r_0 \|_\ps}  + K_1 
\end{equation}
according to \eqref{eq:aggr2}, where
$K_1:= - \frac{1}{\wt\eta_L}
\log \big( \frac{1}{\sqrt{\alpha_*\alpha^*}}\frac{C_0}{C_{\bA^{-1}}} \big) +
\delta_n$ and $\delta_n \in [0,1)$. 

We thus have the following adaptive spectral Galerkin method with
dynamic choice \eqref{aux:1} of the marking parameter
$\theta_n=(1-C_0^2 \|r_n\|_\ps^2/ \|r_0\|_\ps^2)^{1/2}$:

\begin{algotab}
  \> $\textbf{DYN-GAL}(\tol)$\\
  \> set $r_0:=f$, $\Lambda_0:=\emptyset$, $n=-1$\\
  \>  do\\
  \> \> $n \leftarrow n+1$\\
  \> \> $\partial\Lambda_{n}:= \textbf{E-D\"ORFLER} \, \big(\Lambda_n, (1-C_0^2
    \|r_n\|_\ps^2/\| r_0 \|^2_\ps)^{1/2} \big)$\\
  \> \> $\Lambda_{n+1}:=\Lambda_{n} \cup \partial\Lambda_{n}$\\
  \> \> $u_{n+1}:= \textbf{ GAL} \, (\Lambda_{n+1})$\\
  \> \> $r_{n+1}:= \textbf{ RES} \, (u_{n+1})$\\
  \> while $\Vert r_{n+1} \Vert_\ps > \tol ~\Vert r_0 \Vert_\ps$
\end{algotab}
where {\bf GAL} computes the Galerkin solution and {\bf RES} the residual.
The following result shows the quadratic convergence of {\bf DYN-GAL}.
\begin{theorem}[quadratic convergence]
Let the constant $C_0$ of \eqref{aux:1} satisfy
$C_0\le \frac 1 4 \sqrt{\frac{\alpha_*}{\alpha^*}}\frac{\beta_*}{\beta^*}$ and
$C_1 := \frac{\sqrt{\alpha^*}}{2\beta_*\|f\|_\ps}$.
Then the residual $r_n$ of {\bf DYN-GAL} satisfies 
\begin{equation}\label{quadratic-residual-1}
\frac{\|r_{n+1} \|_\ps}{2\|r_{0} \|_\ps} \leq \left(\frac{\|r_n
  \|_\ps}{2\|r_{0} \|_\ps}\right)^2 \quad\forall n\ge0,
\end{equation}
and the algorithm  terminates in finite steps for any
tolerance $\tol$. In addition, two consecutive solutions of {\bf DYN-GAL} satisfy
\begin{equation}\label{quadratic}
  \tvert u - u _{n+1} \tvert \leq  C_1 \tvert u-u_n \tvert^2
   \quad\forall n\ge0.
\end{equation}
\end{theorem}
\begin{proof}
Invoke \eqref{eq:aggr_error_reduct} and \eqref{aux:1} to
figure out that
\begin{equation}\label{quadratic-residual}
\begin{aligned}
\frac{\|r_{n+1} \|_\ps}{\|r_{0} \|_\ps} 
&\leq 
\frac{\sqrt{\alpha^*}}{\beta_*} \frac{\tvert u-  u_{n+1}\tvert}{\|r_{0} \|_\ps} 
\leq
2 \sqrt{\frac{\alpha^*}{\alpha_*}}\frac{\beta^*}{\beta_*}
\sqrt{1-\theta^2_n} \frac{ \| r_n\|_\ps}{\| r_0\|_\ps}
\\
&\leq 
2 C_0 \sqrt{\frac{\alpha^*}{\alpha_*}}\frac{\beta^*}{\beta_*} \left(\frac{ \| r_n\|_\ps}{ \| r_0\|_\ps} \right)^2
\leq
\frac 1 2  \left(\frac{ \| r_n\|_\ps}{ \| r_0\|_\ps} \right)^2,
\end{aligned}
\end{equation}
which implies \eqref{quadratic-residual-1}. We thus
realize that {\bf DYN-GAL} converges quadratically and terminates in
finite steps for any tolerance $\tol$. Finally, combining \eqref{eq:four.2.1bis}
with \eqref{quadratic-residual}, we readily obtain \eqref{quadratic}
upon using $r_0=f$.
\end{proof}

\begin{remark}[super-linear rate]
If the dynamic marking parameter $\theta_n$ is chosen so that
$\sqrt{1-\theta_n^2}=C_0\left(\frac{\| r_n \|_\ps}{\| r_0
  \|_\ps}\right)^\sigma$ for some 
  $\sigma >0$, then we arrive at the rate
$\tvert u - u _{n+1} \tvert \leq  C_1 \tvert u-u_n \tvert^{1+\sigma}$.
\end{remark}
It seems to us that
the quadratic rate \eqref{quadratic} is the first one in adaptivity theory.
The relation \eqref{quadratic-residual-1} reads equivalently
\begin{equation}\label{quadratic-residual-2}
\frac{\|r_{n+1} \|_\ps}{2\|r_{0} \|_\ps} \leq \left(\frac{\|r_{n+1-k} \|_\ps}{2\|r_{0} \|_\ps}\right)^{2^k}\qquad 0\leq k \leq n+1,
\end{equation}
and implies that $\|r_n\|_\ps/\|r_0\|_\ps$ is within machine precision
in about $n=6$ iterations. This fast decay is consistent with spectral
methods. Upon termination,
we obtain the relative error
\[
\tvert u - u_n \tvert \le \frac{\beta^*\sqrt{\alpha^*}}{\beta_*
  \sqrt{\alpha_*}} \tvert u \tvert ~\tol,
\]
because $\|f\|_\ps \le \frac{\sqrt{\alpha^*}}{\beta_*}\tvert u \tvert$.

The algorithm {\bf DYN-GAL} entails exact computation of the
residual $r_n$, which in general has infinitely many terms. We do not
dwell here with inexact or feasible versions of {\bf DYN-GAL} and
refer to \cite{CaNoVe:14a,CaNoVe:14b,CaSiVe:14-2} for a full
discussion which extends to our present setting.

{
\section{Nonlinear Approximation and Gevrey Sparsity Classes}
Given any $v\in V$ we define its {\sl best $N$-term approximation error} as 
$$
E_N(v)= \inf_{\Lambda \subset {\cal K} , 
\ |\Lambda|=N} \Vert v - P_{\Lambda} v \Vert_\phi \;.
$$
We are interested in classifying functions $v$ according to the decay
law of $E_N(v)$ as $N\to\infty$, i.e., according to the ``sparsity'' of their expansions in terms of the basis 
$\{\phi_k\}_{k\in \mathcal{K}}$. 
Of special interest to us is the following exponential Gevrey class.
\begin{definition}[{exponential class of functions}]\label{def:AGev} 
For $\eta >0$ and $0 < t \leq d$,
we denote by ${\mathcal A}^{\eta,t}_G$ the subset of $V$ defined as
$$
{\mathcal A}^{\eta,t}_G { := \Big\{ v \in V \ : 
\ \Vert v \Vert_{{\mathcal A}^{\eta,t}_G}:= 
\sup_{N \geq 0} \, \left( E_N(v) \, {\rm exp}\left(\eta \omega_d^{-t/d}
  N^{t/d} \right) \right) < +\infty \Big\}  \;}
$$
where $\omega_d$ {is the measure of the $d$-dimensional Euclidean}
unit ball $B_d(0,1)$ centered at the origin.
\end{definition}

\begin{definition}[{exponential class of sequences}]\label{def:elpicGev}
{Let $\ell_G^{\eta,t}({\mathcal{K}})$ be the} subset of sequences 
${\bv} \in \ell^{2}({\mathcal{K}})$ so that \looseness=-1
$$
\Vert {\bv} \Vert_{\ell_G^{\eta,t}({\mathcal{K}})} := \sup_{n \geq 1} 
\Big( n^{(1-t/d)/2} {\rm exp}\left(\eta \omega_d^{-t/d} n^{t/d} \right)
|v_n^*| \Big) < +\infty \;,
$$
where {${\bv}^*=(v_n^*)_{n=1}^\infty$} is the non-increasing rearrangement of ${\bv}$.
\end{definition}


The relationship between  ${\mathcal A}^{\eta,t}_G$ and $\ell_G^{\eta,t}({\mathcal{K}})$ is stated
in the following \cite[Proposition 4.2]{CaNoVe:14a}.
\begin{proposition}[{equivalence of exponential classes}]\label{prop:nlg1}
Given a function $v \in V$ and the sequence ${\bv}={(\hat{v}_k \sqrt{d_k})_{k \in {\mathcal{K}}}}$ of its coefficients,
 one has  $v \in {\mathcal A}^{\eta,t}_G$ if and only if
${\bv} \in \ell_G^{\eta,t}({\mathcal{K}})$, with
$$
\|v \|_{{\mathcal A}^{\eta,t}_G} \lesssim 
\Vert {\bv} \Vert_{\ell_G^{\eta,t}({\mathcal{K}})}
\lesssim \| v \|_{{\mathcal A}^{\eta,t}_G}\,.
$$
\end{proposition}

\noindent
For functions $v$ in ${\mathcal A}^{\eta,t}_G$ one can estimate the
minimal cardinality of a set $\Lambda$ such that
$\|v-P_\Lambda v\|_\phi \le\varepsilon$ as follows:
since $\|v-P_{\wt\Lambda} v\|_\phi > \varepsilon$ for any set $\wt\Lambda$ with cardinality
$|\wt\Lambda|=|\Lambda|-1$, we deduce
\begin{equation}\label{bound:optimal}
  \vert \Lambda \vert \leq \omega_d  \left( \frac{1}{\eta} \log 
  \frac{\Vert v \Vert_{{\mathcal A}^{\eta,t}_G}}{\varepsilon} \right)^{d/t} +1. 
\end{equation}

For the analysis of the optimality of our algorithm it is important to investigate 
the sparsity class of the image ${L}v$ for the operator ${L}$ defined
in \eqref{eq:four03}, when
the function $v$ belongs to the sparsity class
$\mathcal{A}^{\eta,t}_G$. Sparsity classes of exponential type for functionals $f \in V^*$ can be defined analogously as above,
using now the best $N$-term approximation error in $V^*$
$$
E^*_N(f)= \inf_{\Lambda \subset {\cal K} , 
\ |\Lambda|=N} \Vert f - P^*_{\Lambda} f \Vert_{\phi^*} \;.
$$

The following result is based on
\cite[Proposition 5.2]{CaNoVe:14a}.

\begin{proposition}[continuity of ${L}$ in $\mathcal{A}^{\eta,t}_G$]\label{propos:spars-res}
Let $L$ be such that the associated stiffness matrix $\mathbf{A}$ satisfies the decay condition \eqref{eq:decay-entries-A}.
Given $\eta>0$ and $t \in (0,d]$, there exist $\bar{\eta}>0$, $\bar{t} \in (0,t]$ and a constant $C_L\ge1$ such that
 \begin{equation}\label{eq:spars11bis}
\Vert Lv \Vert_{{\mathcal A}_G^{\bar{\eta},\bar{t}}} \le {C}_L
\Vert v \Vert_{{\mathcal A}_G^{\eta,t}} \qquad \forall v \in {\mathcal A}_G^{\eta,t}\;. 
\end{equation}
\end{proposition}  
\begin{proof} Let $\bA$ be the stiffness matrix associated with the operator $L$. In \cite{CaNoVe:14a} it is proven 
that if $\mathbf{A}$ is banded with $2p+1$ non-zero diagonals, then the result holds with
$\bar{\eta}= \frac{\eta}{(2p+1)^{t/d}}$ and $\bar{t}= t$; on the other hand, if $\mathbf{A}\in\mathcal{D}_e(\eta_L)$ is dense, but the coefficients $\eta_L$ and $\eta$ satisfy the inequality $\eta< \eta_L  \omega_d^{t/d}$, then the result
holds with $\bar{\eta}= \zeta(t)\eta$ and $\bar{t}= \frac{t}{1+t}$, where ${\zeta(t) = \left( \frac{1+t}{2^d \, \omega_d^{1+t}} \right)^{\frac{t}{d(1+t)}}}$. Finally, if $\eta \geq \eta_L  \omega_d^{t/d}$, we introduce an arbitrary $\hat{\eta}>0$ satisfying $\hat{\eta}< \eta_L  \omega_d^{t/d}$; then the result holds with $\bar{\eta}= \zeta(t)\hat{\eta}$ and $\bar{t}= \frac{t}{1+t}$, since $\Vert v \Vert_{{\mathcal A}_G^{\hat{\eta},t}} \leq \Vert v \Vert_{{\mathcal A}_G^{\eta,t}} $.
\end{proof}

Keeping into account that $\zeta(t) \leq 1$ for $1 \leq d \leq 10$ (see again \cite{CaNoVe:14a}), this result indicates that the residual is expected to belong to a
less favorable sparsity class than the one of the
solution. Counterexamples in \cite{CaNoVe:14a} show that
\eqref{eq:spars11bis} cannot be improved.
}

Finally, we discuss the sparsity class of the residual
$r=r(u_\Lambda)$ for any Galerkin solution $u_\Lambda$.
\begin{proposition}[{sparsity class of the residual}]\label{prop:unif-bound-res-exp}
{Let $\mathbf{A}\in\mathcal{D}_e(\eta_L)$ and
$\bA^{-1} \in\mathcal{D}_e(\bar\eta_L)$, for constants $\eta_L>0$
and $\bar\eta_L\in(0,\eta_L]$ so that \eqref{decay-A} and
\eqref{decayA-1} hold.
If $u \in {\mathcal A}^{\eta,t}_G$ for some $\eta>0$ and $t \in
(0,d]$,
then there exist suitable positive constants $\tilde{\eta} \leq \eta$ and 
$\tilde{t} \leq t$ such that
$r(u_\Lambda) \in {\mathcal A}_G^{\tilde{\eta},\tilde{t}}$ for any index
set $\Lambda$ and
}
$$
\Vert r(u_\Lambda) \Vert_{{\mathcal A}_G^{\tilde{\eta},\tilde{t}}} \lesssim
\Vert u \Vert_{{\mathcal A}^{\eta,t}_G} \;.
$$
\end{proposition}
\begin{proof}
Proposition \ref{propos:spars-res} yields the existence of $\bar{\eta}>0$ and $\bar{t} \in (0,t]$ such that
\begin{equation}\label{eq:spars11ter}
\Vert f \Vert_{{\mathcal A}_G^{\bar{\eta},\bar{t}}}= \Vert Lu \Vert_{{\mathcal A}_G^{\bar{\eta},\bar{t}}} 
\lesssim \Vert u \Vert_{{\mathcal A}_G^{\eta,t}} \;.
\end{equation}
In order to bound $\Vert r(u_\Lambda) \Vert_{{\mathcal A}_G^{\tilde{\eta},\tilde{t}}}$ in terms of $\Vert f \Vert_{{\mathcal A}_G^{\bar{\eta},\bar{t}}}$,
let us write
$$
{\bf r}_\Lambda = \mathbf{A}({\bf u} - {\bf u}_\Lambda ) = {\bf f} - \mathbf{A} {\bf u}_\Lambda \;,
$$
then use $ {\bf u}_\Lambda = ({\widehat{\mathbf{A}}_\Lambda})^{-1} (\mathbf{P}_\Lambda{\bf f})$ from
\eqref{eq:inf-pb-galerkin} to get
$$
{\bf r}_\Lambda = {\bf f} - \mathbf{A} ({\widehat{\mathbf{A}}_\Lambda})^{-1} (\mathbf{P}_\Lambda{\bf f}).
$$
Now, assuming just for simplicity that the indices in $\Lambda$ come first  (this can be realized by a permutation),
we have
$$
 \mathbf{A} = \left(
 \begin{matrix}
 \mathbf{A}_\Lambda & \mathbf{B} \\[5pt]
 \mathbf{B}^T & \mathbf{C}
\end{matrix}
\right)
\qquad
\text{and}
\qquad
\widehat{\mathbf{A}}_\Lambda = \left(
 \begin{matrix}
 \mathbf{A}_\Lambda & \mathbf{O} \\[5pt]
 \mathbf{O}^T & \mathbf{I}
\end{matrix}
\right),
\qquad
\text{whence}
\qquad
(\widehat{\mathbf{A}}_\Lambda)^{-1} = \left(
 \begin{matrix}
 (\mathbf{A}_\Lambda)^{-1} & \mathbf{O} \\[5pt]
 \mathbf{O}^T & \mathbf{I}
\end{matrix}
\right)
$$
Setting $\mathbf{f}=(\mathbf{f}_\Lambda \ \ \mathbf{f}_{\Lambda^c} )^T$, so that $\mathbf{P}_\Lambda{\bf f} = (\mathbf{f}_\Lambda \ \ \mathbf{0} )^T$, we have
\begin{eqnarray*}
\mathbf{A} ({\widehat{\mathbf{A}}_\Lambda})^{-1} (\mathbf{P}_\Lambda{\bf f}) &=&
\left(
 \begin{matrix}
 \mathbf{A}_\Lambda & \mathbf{B} \\[5pt]
 \mathbf{B}^T & \mathbf{C}
\end{matrix}
\right)
\left(
 \begin{matrix}
 (\mathbf{A}_\Lambda)^{-1} & \mathbf{O} \\[5pt]
 \mathbf{O}^T & \mathbf{I}
\end{matrix}
\right)
\left(
 \begin{matrix}
 \mathbf{f}_\Lambda \\[5pt]
\mathbf{0}
\end{matrix}
\right) \\[5pt]
&=& 
\left(
 \begin{matrix}
 \mathbf{A}_\Lambda & \mathbf{B} \\[5pt]
 \mathbf{B}^T & \mathbf{C}
\end{matrix}
\right)
\left(
 \begin{matrix}
 (\mathbf{A}_\Lambda)^{-1}  \mathbf{f}_\Lambda \\[5pt]
\mathbf{0}
\end{matrix}
\right)
\ = \ 
\left(
 \begin{matrix}
 \mathbf{f}_\Lambda \\[5pt]
\mathbf{B}^T  (\mathbf{A}_\Lambda)^{-1}  \mathbf{f}_\Lambda
\end{matrix}
\right)
\end{eqnarray*}
Then,
$$
{\bf r}_\Lambda = \left(
 \begin{matrix}
 \mathbf{0} \\[5pt]
\mathbf{f}_{\Lambda^c} - \mathbf{B^T}  (\mathbf{A}_\Lambda)^{-1}  \mathbf{f}_\Lambda
\end{matrix}
\right)
=
\left(
 \begin{matrix}
 \mathbf{O} & \mathbf{O} \\[5pt]
  -\mathbf{B}^T  (\mathbf{A}_\Lambda)^{-1} & \mathbf{I}
\end{matrix}
\right)  
\left(
 \begin{matrix}
 \mathbf{f}_\Lambda \\[5pt]
\mathbf{f}_{\Lambda^c}
\end{matrix}
\right)
=: \mathbf{R} \mathbf{f}  \;.
$$
Now, since $\mathbf{A}\in\mathcal{D}_e(\eta_L)$ and
$(\bA_\Lambda)^{-1}  \in\mathcal{D}_e(\bar\eta_L)$, it is easily seen that $\mathbf{R}\in\mathcal{D}_e(\tilde{\eta}_L)$ with $\tilde{\eta}_L = \bar{\eta}_L$ if $\bar{\eta}_L< \eta_L$, or $\tilde{\eta}_L < \eta_L$ arbitrary if  $\bar{\eta}_L= \eta_L$.

Finally, we apply Proposition \ref{propos:spars-res} to the operator $R$ defined by the matrix $\mathbf{R}$, 
obtaining the existence of constants $\tilde{\eta}>0$ and $\tilde{t} \in (0,\bar{t}\,]$ such that
$$
\Vert r(u_\Lambda) \Vert_{{\mathcal A}_G^{\tilde{\eta},\tilde{t}}} \lesssim
\Vert f \Vert_{{\mathcal A}_G^{\bar{\eta},\bar{t}}} \;,
$$
whence the result.
\end{proof}

\section{Optimality Properties of {\bf DYN-GAL}}
In this section we derive an exponential rate of convergence for
$\tvert u -u_n\tvert $ in terms of the number of degrees of freedom
$\Lambda_n$ activated by {\bf DYN-GAL} and assess the computational
work necessary to achieve this rate. This is made precise in the
following theorem.
\begin{theorem}[exponential convergence rate] \label{thm5.1}
Let $u\in \mathcal{A}_G^{\eta,t}$ where the Gevrey class
$\mathcal{A}_G^{\eta,t}$ is introduced in Definition \ref{def:AGev}.
Upon termination of {\bf DYN-GAL}, the iterate $u_{n+1}\in V_{\Lambda_{n+1}}$
and set of active coefficients $\Lambda_{n+1}$
satisfy $\tvert u-u_{n+1}\tvert \le \frac{\beta^*}{\sqrt{\alpha_*}} \|f\|_{\phi^*}\tol$ and
\begin{equation}\label{exp-rate}
  |\Lambda_{n+1}| \le \omega_d
  \left( \frac{1}{\eta_*}
  \log \frac{C_* \frac{\|u\|_{\mathcal{A}_G^{\eta,t}}}{\|f\|_{\phi^*}}}{\tol}
  \right)^{d/t_*},
\end{equation}
with parameters $C_*>0$, $\eta_*<\eta$ and $t_*<t$. Moreover, if
the number of arithmetic operations needed to solve a
linear system scales linearly with its dimension, then the workload
$\mathcal{W}_\tol$ of {\bf DYN-GAL} upon completion satisfies
\begin{equation}\label{workload}
  \mathcal{W}_\tol \le \omega_d
  \left(
  \frac{1}{\eta^*} \log
  \frac{C^* \frac{\|u\|_{\mathcal{A}_G^{\eta,t}}}{\|f\|_{\phi^*}}}{\tol^4|\log|\log\tol||^{-1}},
  \right)^{d/t_*}
\end{equation}
%
where $\eta^*<\eta_*$ and $C^*>C_*$ but $t_*$ remains the same as in \eqref{exp-rate}.
\end{theorem}
\begin{proof}
We proceed in several steps.

\medskip
1. {\it Expression of $J(\theta_k)$:}  
Our first task is to simplify the expression \eqref{eq:Jk} for
$J(\theta_k)$, namely to absorb the term $K_1$:
there is $C_2>1$ such that
\begin{equation}\label{J_k}
J(\theta_k)\leq \frac{C_2}{{\wt{\eta}_L}} \Big\vert \log
\frac{\| r_k\|_\ps}{2\|r_0\|_\ps} \Big\vert .
\end{equation}
In fact, if $C_2$ is given by $C_2 = 1+ \frac{\wt\eta_L \max(0,K_1)}{\log 2}$,
then
\[
K_1 \leq \frac{C_2 - 1}{\wt\eta_L} \log 2 \le
\frac{C_2 - 1}{\wt\eta_L} \left\vert \log \frac{\|r_k\|_\ps}{2\|r_0\|_\ps}\right\vert 
\]
because $\|r_k\|_\ps \le \|r_0\|_\ps$ for all $k\ge0$ according to
\eqref{quadratic-residual}. This in turn implies \eqref{J_k}
\[
J(\theta_k) \le\frac{1}{\wt\eta_L}  \left\vert  \log \frac{\|r_k\|_\ps}{2\|r_0\|_\ps} \right\vert
+ \frac{C_2 - 1}{\wt\eta_L} \left\vert \log \frac{\|r_k\|_\ps}{2\|r_0\|_\ps}\right\vert 
= \frac{C_2}{{\wt{\eta}}_L} \left\vert \log
\frac{\| r_k\|_\ps}{2\|r_0\|_\ps} \right\vert .
\]

\medskip
2. {\it Active set $\Lambda_k$:}
We now examine the output $\Lambda_k$ of {\bf E-D\"ORFLER}. 
Employing the minimality of D\"orfler marking and \eqref{doerfler-equiv}, we deduce
\begin{equation}\label{non-lin:dual}
E^*_{\vert\widetilde{\partial\Lambda_k} \vert} (r_k)=\Vert r_k-P_{\widetilde{\partial\Lambda_k}}^* r_k \Vert_\ps \leq \sqrt{1-\theta_k^2}\Vert r_k \Vert_\ps ,
\end{equation}
which clearly implies $E^*_{\vert\widetilde{\partial\Lambda_k}\vert-1}> \sqrt{1-\theta_k^2}\Vert r_k \Vert_\ps$.
The latter inequality, together with the Definition \ref{def:elpicGev}
of $\|r_k\|_{\cA_G^{\bar{\eta},\bar{t}}}$, yields 
%
\begin{equation*}
 \Vert r_k \Vert_{{\mathcal A}^{\bar\eta,\bar t}_G} >
 \sqrt{1-\theta_k^2}\Vert r_k \Vert_\ps\, {\rm exp}\left(\bar{\eta}
 \omega_d^{-\bar{t}/d} (\vert\widetilde{\partial\Lambda}_k\vert-1)^{\bar{t}/d} \right).
\end{equation*}
We note that this, along with $\vert\widetilde{\partial\Lambda}_k\vert\geq 1$,
ensures  
$\frac{\|r_k\|_{\cA_G^{\bar{\eta},\bar{t}}}}{\sqrt{1-\theta_k^2}\|r_k\|_\ps}>1$ whence
\[
|\widetilde{\partial\Lambda_k}| \leq \omega_d \left(
\frac{1}{\bar{\eta}} \log
\frac{\|r_k\|_{\cA_G^{\bar{\eta},\bar{t}}}}{\sqrt{1-\theta_k^2}\|r_k\|_\ps}\right)^{d/\bar{t}}+1.
\]

We now recall the membership of the residual $r_k := r(u_k)$ to the Gevrey
class $\mathcal{A}_G^{\bar \eta, \bar t}$ for all $k\ge 0$,
established in Proposition \ref{prop:unif-bound-res-exp}: there exists
$C_3>0$ independent of $k$ and $u$ such that
\begin{equation*}
\|r_k\|_{\mathcal{A}_G^{\bar \eta, \bar t}} \le C_3 \|u\|_{\mathcal{A}_G^{\eta, t}}.
\end{equation*}
Combining this with the dynamic marking \eqref{aux:1} implies
\[
\frac{\|r_k\|_{\mathcal{A}_G^{\bar \eta, \bar t}}}{\sqrt{1-\theta_k^2} \|r_k\|_\ps}
= \frac{\|f\|_\ps \|r_k\|_{\mathcal{A}_G^{\bar \eta, \bar t}}}{C_0\|r_k\|_\ps^2}
\le \frac{C_4\|f\|_\ps \|u\|_{\mathcal{A}_G^{\eta, t}}}{\|r_k\|_\ps^2},
\]
with $C_4=C_3/C_0$, whence
\[
|\widetilde{\partial\Lambda_k}| \leq \omega_d \left(
\frac{1}{\bar{\eta}} \log
\frac{C_4 \|f\|_\ps \|u\|_{\mathcal{A}_G^{\eta, t}}}{\|r_k\|_\ps^2}\right)^{d/\bar{t}}+1.
\]
We let $C_5$ satisfy
$1=\omega_d \big(\frac{1}{\bar\eta} \log C_5 \big)^{d/\bar t}$, and
use that $\bar t < t \le d$, to obtain the simpler expression
\begin{equation}\label{aux-card:1}
|\widetilde{\partial\Lambda}_k| \leq
\omega_d \left(
\frac{1}{\bar{\eta}} \log
\frac{C_4 C_5 \|f\|_\ps \|u\|_{\mathcal{A}_G^{\eta, t}}}{\|r_k\|_\ps^2}\right)^{d/\bar{t}}.
\end{equation}

On the other hand, in view of \eqref{J_k},
the enrichment step \eqref{card-Lambda} of {\bf E-D\"ORFLER} yields
\begin{equation}\label{aux-card:2}
|\partial\Lambda_k| \le C_6 \omega_d J(\theta_k)^d |\widetilde{\partial\Lambda}_k|
\le \frac{C_6C_2^d \omega_d^2}{\wt\eta_L^{d}\bar\eta^{d/\bar t}}
\left\vert \log \frac{\| r_k\|_\ps}{2\|r_0\|_\ps} \right\vert^d 
\left (
\log \frac{ C_4 C_5\|f\|_\ps \|u\|_{\cA_G^{\eta,t}}}{\|r_k\|_\ps^2}\right)^{d/\bar{t}}.
\end{equation}
We exploit the quadratic convergence \eqref{quadratic-residual-1} to
write for $k\le n$
\begin{equation}\label{aux-card:4}
\left| \log\frac{\|r_k\|_\ps}{2\|r_0\|_\ps}\right|^d \le
2^{d(k-n)} \left| \log \frac{\|r_n\|_\ps}{2\|r_0\|_\ps} \right|^d 
=2^{d(k-n)} \left( \log \frac{2\| f\|_\ps}{\|r_n\|_\ps} \right)^d. 
\end{equation}
Using the bound $\|f\|_{\phi^*}\le \|f\|_{\cA_G^{\bar\eta,\bar{t}}}
= \|r_0\|_{\cA_G^{\bar\eta,\bar{t}}} \leq C_3 \|u\|_{\cA_G^{\eta,t}}$ in the two previous inequalities
yields
\[
\log \frac{ C_4 C_5\|f\|_\ps \|u\|_{\cA_G^{\eta,t}}}{\|r_k\|_\ps^2} \leq
2 \log \frac{ C_7 \|u\|_{\cA_G^{\eta,t}}}{\|r_k\|_\ps} \qquad \text{and} \qquad 
 \log \frac{2\| f\|_\ps}{\|r_n\|_\ps} \leq  \log \frac{2C_3 \|u\|_{\cA_G^{\eta,t}}}{\|r_n\|_\ps}, 
 \]
 with $C_7=(C_3C_4C_5)^{1/2}$. Introducing the constants 
\[
C_8 = \max (2C_3, C_7),
\qquad
\frac1{\hat \eta} = \frac{2^{d/\bar t}C_6C_2^d \omega_d}{\wt\eta_L^{d}\bar\eta^{d/\bar t}},
\qquad
t_* = \frac{\bar t}{1+\bar t},
\]
the derived upper bound for $|\partial\Lambda_k|$ can be simplified as follows:
\[
|\partial\Lambda_k| \le 2^{d(k-n)}\frac{\omega_d}{\hat\eta}
\left( \log \frac{C_8 \|u\|_{\cA_G^{\eta,t}}}{\|r_n\|_\ps} \right)^{d/t_*}.
\]

Recalling now that $|\Lambda_0|=0$ and for $n\ge0$
\begin{equation*}
\vert \Lambda_{n+1}\vert = \sum_{k=0}^n \vert \partial \Lambda_k \vert 
\end{equation*}
we have 
\begin{equation*}
  \vert \Lambda_{n+1}\vert \leq \frac{\omega_d}{\hat\eta}
  \left( \sum_{k=0}^{n} 2^{d(k-n)}\right) 
  \left ( \log\frac{C_8 \|u\|_{\cA_G^{\eta,t}}}{\|r_n\|_\ps} \right)^{d/t_*}.
\end{equation*}
This can be written equivalently as
\begin{equation}\label{card:Lambda_n}
\vert \Lambda_{n+1}\vert \leq \omega_d \left( \frac1{{\eta_*}} \log \frac{C_8 \|u\|_{\cA_G^{\eta,t}}}{\Vert r_n \Vert_\ps}  \right)^{d/t_*}
\end{equation}
with $\eta_*=\left(\frac{\hat\eta}{\sum_{k=1}^\infty 2^{-dk}}\right)^{t_*/d}$. At last, we make use of
$\|r_n\|_{\phi^*} > \tol \|f\|_{\phi^*}$ to 
get the desired estimate \eqref{exp-rate} with $C_*=C_8$.

\medskip
3. {\it Computational work:}
Let us finally discuss the total computational work $\mathcal{W}_\tol$
of {\bf DYN-GAL}. We start with some useful notation.
We set $\delta_n:=\frac{\|r_n\|_{\ps}}{\|r_0\|_\ps}$ for
$n\ge0$ and $\varepsilon_{\ell+1}=\varepsilon_\ell^2$ for $\ell\geq 1$
with $\varepsilon_1=\frac 1 2$. We note that there exists an integer $L>0$
such that $\varepsilon_{L+1} < \varepsilon \le \varepsilon_{L}$ with
$\varepsilon$ being the tolerance of {\bf DYN-GAL}.
In addition, we observe that for every iteration $n>0$ of {\bf
DYN-GAL} there exists $\ell>0$ such that 
$\delta_n \in I_\ell:=(\varepsilon_{\ell+1}, \varepsilon_\ell]$ and that for
each interval $I_\ell$
there exists at most one $\delta_n \in I_\ell$ because
$\delta_{n+1}\le\frac12 \delta_n^2$ according to \eqref{quadratic-residual-1}. 
Finally, to each interval $I_\ell$ we associate the following computational work
$W_\ell$ to find and store $u_{n+1}$
\[
W_\ell=
\begin{cases}
C_\# \vert \Lambda_{n+1}\vert & \text{\rm if there exists $n$ such that }
\delta_n\in I_\ell \\
0 & \text{\rm otherwise}\;.
\end{cases}
\]
This assumes that the number of arithmetic operations needed to solve the
linear system for $u_n$ scales linearly with its dimension and $C_\#$
is an absolute constant that may depend on the specific solver.
The total computational work of {\bf DYN-GAL} is bounded by 
\[
\mathcal{W}_\tol = \sum_{\ell=1}^L W_\ell.
\]
We now get a bound for $\mathcal{W}_\tol$.
In view of \eqref{exp-rate} we have
\[
W_\ell \leq C_\# \omega_d \left( \frac1{{\eta_*}} \log \frac{C_*
  \frac{\|u\|_{\cA_G^{\eta,t}}}{\|f\|_{\phi^*}}}{\varepsilon_{\ell+1} }  \right)^{d/t_*} =
  C_\# \omega_d \left( \frac1{{\eta_*}} \log \frac{C_*
  \frac{\|u\|_{\cA_G^{\eta,t}}}{\|f\|_{\phi^*}}}{\varepsilon_1^{2^{\ell}}}  \right)^{d/t_*}.
\]
Therefore, upon adding over $\ell$ and using that $d/t_*\ge1$, we
obtain
\begin{align*}
\mathcal{W}_\tol 
\leq \frac{C_\#\omega_d}{\eta_*^{d/t_*}}
  \left( \sum_{\ell=1}^L \log C_* \frac{\|u\|_{\cA_G^{\eta,t}}}{\|f\|_{\phi^*}}
  + \sum_{\ell=1}^L \log \varepsilon_1^{-2^{\ell}} \right)^{d/t_*}
  \leq \frac{C_\#\omega_d}{\eta_*^{d/t_*}}
  \left( \log \frac{L C_* \frac{\|u\|_{\cA_G^{\eta,t}}}{\|f\|_{\phi^*}}}{\varepsilon_1^{2^{L+1}}}  \right)^{d/t_*}.
\end{align*}
Since $\tol\le \varepsilon_{L}=\varepsilon_1^{2^{L-1}}$, we deduce
$\tol^4\leq \varepsilon_1^{2^{L+1}}$ and 
$L \le \frac{\log \frac{|\log \tol|}{\log 2}}{\log 2} +1
\le C_9 \log|\log\tol|$. Inserting this bound in the preceding
expression yields
\begin{equation}
\mathcal{W}_\tol \leq \omega_d \left( \frac1{\eta^*} \log
\frac{C^* \frac{\| u\|_{\cA_G^{\eta,t}}}{\|f\|_{\phi^*}}}{\tol^4 |\log \vert\log
  \tol \vert \vert^{-1}}  \right)^{d/t_*}.
\end{equation}
with
\[
\eta^* = \frac{\eta_*}{C_\#^{t_*/d}},
\quad
C^* = C_9 C_*,
\]
which is the asserted estimate \eqref{workload}. The proof is thus complete.
\end{proof}

\begin{remark} Note that the bound on the workload, given in \eqref{workload}, is at most an absolute multiple of the bound, given in \eqref{exp-rate}, on the number of active coefficients.
\end{remark}

\begin{remark}[super-linear convergence]
If $\sqrt{1-\theta_n^2}=C_0\left(\frac{\| r_n \|_\ps}{\| r_0
  \|_\ps}\right)^\sigma$ with $\sigma >0$, then \eqref{exp-rate} still holds with the same parameters $\eta_*,t_*$.
\end{remark}
\begin{remark}[algebraic class]
Let us consider the case when $u$ belongs to the algebraic class 
\[
  \mathcal{A}_B^s:= { \Big\{ v \in V \ : \ \Vert v \Vert_{{\mathcal A}^{s}_B}:= 
    \sup_{N \geq 0} \, E_N(v) \, \big(N+1\big)^{s/d}  < +\infty \Big\}},
\]
which is related to Besov regularity.  We can distinguish two cases:
\begin{enumerate}
\item
$\mathbf{A}$ belongs to an exponential class but the residuals belong to an algebraic class; 
\item $\mathbf{A}$ belongs to an algebraic class $\mathcal{D}_a(\eta_L)$,  i.e. there exists a constant $c_L>0$ 
such that its elements satisfy $| a_{\ell,k} | \leq  c_L (1+ \vert \ell - k \vert )^{-\eta_L}$, and the residual belongs to an algebraic class.
\end{enumerate}
We now study the optimality properties of {\bf DYN-GAL} for
these two cases. Let us first
observe that whenever the residuals belong to the algebraic class $\mathcal{A}_B^s$, the bound 
\eqref{aux-card:1} becomes
\begin{equation}\label{aux-card:3}
|\widetilde{\partial\Lambda}_k| \lesssim \| u -u_k\|^{-2d/s}.
\end{equation}
This results from the dynamic marking  \eqref{aux:1} together
with \eqref{non-lin:dual} in the algebraic case.

Let us start with Case $1.$ Using  \eqref{aux-card:3}  and
\eqref{J_k}, which is still valid here as we assume that $\mathbf{A}$
belongs to an exponential class, the bound \eqref{aux-card:2} is replaced by
\begin{equation}
|{\partial\Lambda}_k| \lesssim \| u -u_k\|^{-2d/s}\Big\vert \log
\frac{\| r_k\|_\ps}{2\|r_0\|_\ps} \Big\vert^d \lesssim  \| u
-u_k\|^{-2d/s}2^{d(k-n)} \left( \log \frac{2\| f\|_\ps}{\|r_n\|_\ps}
\right)^d ,
\end{equation}
where in the last inequality we have employed \eqref{aux-card:4}. Hence, we have 
\begin{eqnarray}
\vert \Lambda_{n+1}\vert &=& \sum_{k=0}^n \vert \partial \Lambda_k \vert 
\lesssim  \left( \log \frac{2\| f\|_\ps}{\|r_n\|_\ps} \right)^d \sum_{k=0}^n \| u- u_n\|^{-2\frac d s 2^{k-n}} 2^{d(k-n)}\nonumber\\
&\lesssim&   \left( \log \frac{2\| f\|_\ps}{\|r_n\|_\ps} \right)^d \| u- u_n\|^{-2\frac d s}  \sum_{k=0}^n 2^{d(k-n)}\nonumber\\
&\lesssim&   \left( \log \frac{2\| f\|_\ps}{\|r_n\|_\ps} \right)^d\| u- u_n\|^{-2\frac d s}   \lesssim  \left( \log \frac{2\| f\|_\ps}{\|u-u_n\|} \right)^d \|u-u_{n+1}\|^{- \frac d s }\nonumber
\end{eqnarray}
where in the last inequality we have employed the quadratic convergence of  {\bf DYN-GAL}.
The above result implies that {\bf DYN-GAL} is optimal
for Case 1 (up to a logarithmic factor).

Let us now consider Case $2.$ Since $\mathbf{A}$ belongs to an algebraic class, \eqref{J_k} is replaced by 
\begin{equation*}\label{aux-card:5}
J(\theta_k)\eqsim \| u-u_k\|^{-1/s}.
\end{equation*}
Thus, the bound \eqref{aux-card:2} is replaced by
\begin{equation*}
|{\partial\Lambda}_k| \lesssim \| u -u_k\|^{-3d/s},
\end{equation*}
which implies 
\begin{equation*}
\vert \Lambda_{n+1}\vert = \sum_{k=0}^n \vert \partial \Lambda_k \vert 
\lesssim  \|u-u_n\|^{-3 \frac d s }
\lesssim  \|u-u_{n+1}\|^{-\frac 3 2 \frac d s }
\end{equation*}
where in the last inequality we have again employed the
quadratic convergence of  {\bf DYN-GAL}.
This result is \emph{not} optimal for Case 2,
due to the factor 2/3 multiplying $s$ in the exponent.
We recall that the algorithm {\bf FA-ADFOUR} of
\cite{CaNoVe:14a} is similar to {\bf DYN-GAL}  
but with static marking parameter $\theta$. The theory of {\bf FA-ADFOUR}
requires neither a restriction on $\theta$ nor coarsening and is proven to
be optimal in the algebraic case; see \cite[Theorem $7.2$]{CaNoVe:14a}.

\end{remark}

\bigskip 
\noindent {\bf Acknowledgements.}
%
The first and the fourth authors are partially supported by the Italian research grant  {\sl Prin 2012}  2012HBLYE4  ``Metodologie innovative nella modellistica differenziale numerica''. 
The first author is also partially supported by Progetto INdAM-GNCS 2015 ``Tecniche di riduzione computazionale per problemi di fluidodinamica e interazione fluido-struttura". 
The second author is partially supported by NSF grants DMS-1109325 and DMS-1411808.
The fourth author is also partially supported by Progetto INdAM - GNCS 2015 ``Non-standard numerical methods for geophysics". 


\end{document}